\documentclass{article}
\usepackage{macros}
\usepackage{xurl}

\bibliography{references}

\title{Completely parabolic surfaces are Veech}
\author{Inbo Gottlieb Fenves}
\date{\today}

\begin{document}

\maketitle

\begin{abstract}
    \noindent We prove that completely parabolic translation surfaces are Veech, answering a question of Smillie-Weiss \cite{SmiWei2}.
\end{abstract}

A translation surface is \textit{Veech} if its stabilizer under the $\SL_2(\R)$-action is a lattice. Veech surfaces exhibit remarkable dynamical properties, and in particular obey the \textit{Veech dichotomy}  \cite{Vee}: straight line flow in any direction is either uniquely ergodic or decomposes into finitely many cylinders with commensurable moduli. There has been considerable interest in giving alternate characterizations of Veech surfaces: see e.g. Vorobets \cite{Vor}, Smillie-Weiss \cite{SmiWei}, Chaika-Lindsey \cite{ChaLin}, and Forester-Tang-Tao \cite{ForTanTao}.

A \textit{periodic direction} on a translation surface is a direction in which the straight line flow fully decomposes into cylinders; it is \textit{parabolic} if the cylinders have pairwise commensurable moduli. A surface is \textit{completely periodic} if every direction containing a cylinder is periodic, and \textit{completely parabolic} if every such direction is parabolic. By Lanneau-Nguyen-Wright \cite{LanNguWri}, all completely parabolic surfaces with rank-1 orbit closures are Veech. Without the parabolicity condition, the Veech dichotomy does not characterize Veech surfaces: see Smillie-Weiss \cite{SmiWei1} and Lanneau-Nguyen \cite{LanNgu} for counterexamples. Our main theorem answers a question of Smillie-Weiss (\cite[Question 5]{SmiWei2}, \cite[Problem 10.2]{Wri}), establishing the ``converse Veech dichotomy.''

\begin{thm}\label{thm:mainthm}
    A translation surface is Veech if and only if it is completely parabolic.
\end{thm}

The implication that Veech implies complete parabolicity was established in \cite{Vee}, so it suffices to show that any completely parabolic surface is Veech. We begin with the following lemma, which uses extensively results of Wright \cite{Wri} on cylinder deformations.

\begin{lemma}\label{lemma:finiteareas}
    A translation surface is Veech if and only if its set of cylinder areas is finite.
\end{lemma}
\begin{proof}
    If a translation surface $x$ is Veech, then finiteness of cylinder areas follows from \cite{Vee}. Now, suppose $x$ is not Veech, and let $\Mc = \overline{\GL_2^+(\R)x}$ be its orbit-closure; by assumption, $\Mc$ is not a Veech curve. We claim that it suffices to show that there exist surfaces in $\Mc^1$ with cylinders of arbitrarily small area. Indeed, if this holds, then for each $\epsilon > 0$, we may choose a surface $y \in \Mc^1$ with a cylinder $C$ of area at most $\epsilon$; since $x$ has dense $\SL_2(\R)$-orbit in $\Mc^1$ up to rescaling, we can find some $g \in \SL_2(\R)$ with $gx$ arbitrarily close to $y$. Since cylinders persist in open sets, we can find a cylinder $C'$ on $gx$ of area at most $2\epsilon$, and since $g$ preserves area this cylinder is also area at most $2\epsilon$ on $x$; thus $x$ has arbitrarily low cylinder areas, and in particular this set is infinite.
    
    Suppose first $\rk \Mc \ge 2$, and choose a surface $x_0 \in \Mc$ with two disjoint $\Mc$-equivalence classes of horizontal cylinders $\Cc,\Dc$ by \cite[Corollary 9.1]{Wri}. By \cite[Proof of Theorem 1.1]{Wri}, the flow $a_{\Cc}^t$ which stretches the verticals of $\Cc$ by a factor of $e^{t}$ and acts trivially on the rest of the surface lies in $T_{x_0}\Mc$, and so by choosing $t \gg 1$ we can make the $\Cc$-cylinders on $x_t = a_{\Cc}^{-t}x_0$ have arbitrarily small area. We may then apply $a_{\Dc}^s$-flow until $x_t' = a_\Dc^sx_t \in \Mc^1$ has unit-area, and has arbitrarily low-area cylinders.

    Now, suppose $\rk \Mc = 1$, and $\rel \Mc \ge 1$. Choose a surface $x_0 \in \Mc^1$ with the maximal number of horizontal cylinders; note that the rank-1 assumption implies $x_0$ is horizontally periodic by \cite[Theorem 1.5]{Wri}. By \cite[Lemma 8.6]{Wri}, we can find a nonzero real rel vector $v \in T_{x_0}\Mc \cap \ker p \cap H^1(X,\Sigma;\R)$ which acts trivially on every horizontal saddle connection; then $x_t = x_0 + itv$ remains in $\Mc^1$ so long as it is well-defined. If $t_* > 0$ is the first value for which a cylinder degenerates, then for $0 < t < t_*$, at least one of the cylinders on $x_t$ will have area of order $t_* - t$, which can be arbitrarily small.
\end{proof}

\begin{lemma}\label{lemma:unitareadeduction}
    Suppose $x$ is completely parabolic. Then there exists some $g \in \GL_2^+(\R)$, some totally real number field $\kay$, and some integer $N$ so that $y=gx$ has unit area, and moreover $\Lambda_y^{\rel} = \hol_y(H_1(X,\Sigma;\Z))$ lies in $N^{-1}(\Oc_\kay + i\Oc_\kay)$.
\end{lemma}
\begin{proof}
    Since $x$ is completely parabolic, its Veech group $\Gamma = \Stab_{\SL_2(\R)}(x)$ contains two independent parabolics, and thus also contains a hyperbolic element. By McMullen \cite[Theorems 9.4-9.5, Corollary 9.6]{McM1}, we may choose some $h \in\GL_2^+(\R)$ so that the relative periods $\Lambda_{hx}^{\rel}$ lie in $\kay + i\kay$, where $\kay = \Q(\tr \Gamma)$ is the \textit{trace field} of $x$; by Hubert-Lanneau \cite[Theorem 1.1]{HubLan}, $\kay$ is totally real. Since area can be written as a quadratic form in the absolute periods of $hx$ with $\Z$-coefficients, it follows that $\Area(hx) \in \kay$, and so applying the diagonal matrix $h' = \diag(\Area(hx)^{-1},1) \in \GL_2^+(\kay)$, we see that $y = h'hx$ has unit-area and relative periods in $\kay + i\kay$. Since $\Lambda_y^{\rel}$ is finitely generated, it follows that actually $\Lambda_y^{\rel} \subseteq N^{-1}(\Oc_\kay + i\Oc_\kay)$ for some $N$.
\end{proof}

We will also use the following number-theoretic lemma.

\begin{lemma}\label{lemma:numbertheory}
    Suppose $\kay$ is a totally real number field and $R  > 0$. Then the set of $\alpha \in \Oc_\kay$ with $|\sigma(\alpha)| \le R$ for each embedding $\sigma \colon \kay \hookrightarrow \R$ is a finite set.
\end{lemma}
\begin{proof}
    This follows from the fact that $\Oc_\kay$ is discrete in the \textit{Minkowski embedding} $\kay \hookrightarrow \R^d$.
\end{proof}

We are now in place to prove Theorem \ref{thm:mainthm}.

\begin{proof}[Proof of Theorem \ref{thm:mainthm}]
    Note that $x$ is Veech if and only if $gx$ is Veech for each $g \in \GL_2^+(\R)$, so we may replace $x$ with $y$ as in Lemma \ref{lemma:unitareadeduction}. Now, choose a cylinder direction on $y$, whose cylinders $C_i$ have circumferences $c_i$, heights $h_i$, and moduli $\mu_i = h_i/c_i$. By parabolicity, there exist $q_{ij} \in \Q_+$ for which $\mu_j = q_{ij}\mu_i$. We have
    \begin{equation*}
        1 = \sum_j c_jh_j = \sum_j \mu_jc_j^2 = \mu_i \sum_j q_{ij}c_j^2.
    \end{equation*}
    Note that $c_j^2$ is a sum of squares in $\kay$, hence \textit{totally positive}.\footnote{An element $\alpha$ of a totally real number field $\kay$ is \textit{totally positive} if $\sigma(\alpha) > 0$ for every embedding $\sigma \colon \kay \hookrightarrow \R$.} Thus $\mu_i^{-1} = \sum_j q_{ij}c_j^2$ and $c_ih_i = \mu_ic_i^2$ are also totally positive. Since $\sum_j \sigma(c_jh_j) = 1$ for each embedding $\sigma \colon \kay \hookrightarrow \R$, it follows that $0 < \sigma(c_ih_i) \le 1$. But $c_ih_i = \Area(C_i)$ is also expressible as the determinant of some pair in $\Lambda_y^{\rel}$, and so $c_ih_i \in N^{-2}\Oc_\kay$. Applying Lemma \ref{lemma:numbertheory} with $R = N^2$ shows that the set of cylinder areas is finite, so we are done by Lemma \ref{lemma:finiteareas}.
\end{proof}

\textbf{Acknowledgments:} Research partially funded by NSF RTG grant No. DMS-2230900. I would like to thank my advisor Paul Apisa for suggesting the proof of Lemma \ref{lemma:finiteareas} in the rank-1 case. I would also like to thank Curtis McMullen and Barak Weiss for helpful comments on an early draft, substantially improving the exposition.

\textbf{AI Disclosure:} The first proof of Theorem \ref{thm:mainthm} was found by ChatGPT. The author subsequently produced alternate proofs of each of the intermediate results. All writing is human-generated.

\printbibliography

@article{LanNguWri, title={Finiteness of Teichmüller Curves in Non-Arithmetic Rank 1 Orbit Closures}, volume={139}, ISSN={1080-6377}, url={https://muse.jhu.edu/article/677444}, DOI={10.1353/ajm.2017.0036}, number={6}, journal={American Journal of Mathematics}, author={Lanneau, Erwan and Nguyen, Duc-Manh and Wright, Alex}, year={2017}, pages={1449–1463}, language={en} }

@article{HubLan, title={Veech groups without parabolic elements}, volume={133}, ISSN={0012-7094, 1547-7398}, url={https://projecteuclid.org/journals/duke-mathematical-journal/volume-133/issue-2/Veech-groups-without-parabolic-elements/10.1215/S0012-7094-06-13326-4.full}, DOI={10.1215/S0012-7094-06-13326-4}, abstractNote={We prove that a translation surface that has two transverse parabolic elements has a totally real trace field. As a corollary, nontrivial Veech groups that have no parabolic elements do exist. The proof follows Veech’s viewpoint on Thurston’s construction of pseudo-Anosov diffeomorphisms}, number={2}, journal={Duke Mathematical Journal}, author={Hubert, Pascal and Lanneau, Erwan}, year={2006}, month=june, pages={335–346} }

@article{SmiWei1, title={Veech’s dichotomy and the lattice property}, volume={28}, rights={https://www.cambridge.org/core/terms}, ISSN={0143-3857, 1469-4417}, url={https://www.cambridge.org/core/product/identifier/S0143385708000114/type/journal_article}, DOI={10.1017/S0143385708000114}, number={6}, journal={Ergodic Theory and Dynamical Systems}, author={Smillie, John and Weiss, Barak}, year={2008}, month=dec, pages={1959–1972}, language={en} }

@incollection{SmiWei2,
  author    = {Smillie, John and Weiss, Barak},
  title     = {Finiteness results for flat surfaces: a survey and problem list},
  booktitle = {Partially Hyperbolic Dynamics, Laminations, and {Teichm{\"u}ller} Flow},
  editor    = {Forni, Giovanni and Lyubich, Mikhail and Pugh, Charles and Shub, Michael},
  series    = {Fields Institute Communications},
  volume    = {51},
  pages     = {125--137},
  publisher = {American Mathematical Society},
  address   = {Providence, RI},
  year      = {2007},
  doi       = {10.1090/fic/051/06}
}

@article{Vor, title={Planar structures and billiards in rational polygons: the Veech alternative}, volume={51}, ISSN={0036-0279, 1468-4829}, url={https://www.mathnet.ru/eng/rm1012}, DOI={10.1070/RM1996v051n05ABEH002993}, number={5}, journal={Russian Mathematical Surveys}, author={Vorobets, Ya B}, year={1996}, month=oct, pages={779–817} }

@article{SmiWei, title={Characterizations of lattice surfaces}, volume={180}, rights={http://www.springer.com/tdm}, ISSN={0020-9910, 1432-1297}, url={http://link.springer.com/10.1007/s00222-010-0236-0}, DOI={10.1007/s00222-010-0236-0}, number={3}, journal={Inventiones mathematicae}, author={Smillie, John and Weiss, Barak}, year={2010}, month=june, pages={535–557}, language={en} }

@article{ChaLin, title={Horocycle flow orbits and lattice surface characterizations}, volume={39}, rights={https://www.cambridge.org/core/terms}, ISSN={0143-3857, 1469-4417}, url={https://www.cambridge.org/core/product/identifier/S0143385717000827/type/journal_article}, DOI={10.1017/etds.2017.82}, number={06}, journal={Ergodic Theory and Dynamical Systems}, author={Chaika, Jon and Lindsey, Kathryn}, year={2019}, month=june, pages={1441–1461}, language={en} }

@article{LanNgu, title={Complete periodicity of Prym eigenforms}, volume={49}, ISSN={0012-9593, 1873-2151}, url={http://smf4.emath.fr/Publications/AnnalesENS/4_49/html/ens_ann-sc_49_87-130.php}, DOI={10.24033/asens.2277}, number={1}, journal={Annales scientifiques de l’École normale supérieure}, author={Lanneau, Erwan and Nguyen, Duc-Manh}, year={2016}, pages={87–130} }

@article{Vee, title={Teichm\"uller curves in moduli space, Eisenstein series and an application to triangular billiards}, volume={97}, rights={http://www.springer.com/tdm}, ISSN={0020-9910, 1432-1297}, url={http://link.springer.com/10.1007/BF01388890}, DOI={10.1007/BF01388890}, number={3}, journal={Inventiones Mathematicae}, author={Veech, W. A.}, year={1989}, month=oct, pages={553–583}, language={de} }

@article{Wri, title={Cylinder deformations in orbit closures of translation surfaces}, volume={19}, ISSN={1364-0380, 1465-3060}, url={http://msp.org/gt/2015/19-1/p09.xhtml}, DOI={10.2140/gt.2015.19.413}, number={1}, journal={Geometry and Topology}, author={Wright, Alex}, year={2015}, month=feb, pages={413–438}, language={en} }

@article{McM1, title={Teichmüller geodesics of infinite complexity}, volume={191}, ISSN={0001-5962}, url={http://projecteuclid.org/euclid.acta/1485891606}, DOI={10.1007/BF02392964}, number={2}, journal={Acta Mathematica}, author={McMullen, Curtis T.}, year={2003}, pages={191–223}, language={en} }

@article{ForTanTao, title={Veech surfaces and simple closed curves}, volume={223}, ISSN={0021-2172, 1565-8511}, url={http://link.springer.com/10.1007/s11856-017-1617-5}, DOI={10.1007/s11856-017-1617-5}, number={1}, journal={Israel Journal of Mathematics}, author={Forester, Max and Tang, Robert and Tao, Jing}, year={2018}, month=feb, pages={323–342}, language={en} }
\end{document}